\documentclass{article}

\usepackage[a4paper,left=42mm,right=42mm]{geometry}

\usepackage[utf8]{inputenc}

\usepackage{amssymb,amsmath,amsthm}
\theoremstyle{plain} 
    \newtheorem{theorem}{Theorem}[section]
    \newtheorem*{theorem*}{Theorem}
    \newtheorem{corollary}[theorem]{Corollary}
    \newtheorem{lemma}[theorem]{Lemma}
    \newtheorem{proposition}[theorem]{Proposition}
\theoremstyle{definition}
    \newtheorem{definition}[theorem]{Definition}

\theoremstyle{remark}

\theoremstyle{plain}

\usepackage{IEEEtrantools}

\newcommand{\range}{\operatorname{range}}

\newcommand{\satisfies}{\models}

\newcommand{\crit}{\operatorname{crit}}

\newcommand{\ZFC}{\mathrm{ZFC}}
\newcommand{\ZF}{\mathrm{ZF}}

\newcommand{\OR}{\mathrm{OR}}

\newcommand{\dbar}[1]{\bar{\bar #1}}
\newcommand{\cof}{\operatorname{cof}}

\newcommand{\AC}{\mathrm{AC}}

\usepackage{comment}

\usepackage{soul}

\usepackage{enumitem}
\setlist[enumerate]{label={\upshape(\roman*)},noitemsep}

\usepackage{tabularx}
\newcolumntype{C}{>{\centering\arraybackslash$}X<{$}}
\newcolumntype{M}{>{\centering\arraybackslash$}p{.35cm}<{$}}

\usepackage{todonotes}

\usepackage{biblatex}
\renewbibmacro{in:}{}
\usepackage{caption}

\makeatletter
\DeclareFontFamily{OMX}{MnSymbolE}{}
\DeclareSymbolFont{MnLargeSymbols}{OMX}{MnSymbolE}{m}{n}
\SetSymbolFont{MnLargeSymbols}{bold}{OMX}{MnSymbolE}{b}{n}
\DeclareFontShape{OMX}{MnSymbolE}{m}{n}{
    <-6>  MnSymbolE5
   <6-7>  MnSymbolE6
   <7-8>  MnSymbolE7
   <8-9>  MnSymbolE8
   <9-10> MnSymbolE9
  <10-12> MnSymbolE10
  <12->   MnSymbolE12
}{}
\DeclareFontShape{OMX}{MnSymbolE}{b}{n}{
    <-6>  MnSymbolE-Bold5
   <6-7>  MnSymbolE-Bold6
   <7-8>  MnSymbolE-Bold7
   <8-9>  MnSymbolE-Bold8
   <9-10> MnSymbolE-Bold9
  <10-12> MnSymbolE-Bold10
  <12->   MnSymbolE-Bold12
}{}
\DeclareMathDelimiter{\ulcorner}
    {\mathopen}{MnLargeSymbols}{'036}{MnLargeSymbols}{'036}
\DeclareMathDelimiter{\urcorner}
    {\mathclose}{MnLargeSymbols}{'043}{MnLargeSymbols}{'043}
\makeatother
\title{Reinhardt Cardinals and Eventually Dominating Functions}

\author{Marwan Salam Mohammd
\\
    \texttt{marwan.mizuri@gmail.com}
}

\begin{document}

\maketitle
%\tableofcontents

\begin{abstract}
    We prove a result concerning elementary embeddings of the set-theoretic universe into itself (Reinhardt embeddings) and functions on ordinals that ``eventually dominate'' such embeddings. We apply that result to show the existence of elementary embeddings satisfying some strict conditions and that are also reminiscent of extendibility in a more local setting. Building further on these concepts, we make precise the nature of some large cardinals whose existence under Reinhardt embeddings was proven by Gabriel Goldberg in his paper ``Measurable Cardinals and Choiceless Axioms.'' Finally, these ideas are used to present another proof of the Kunen inconsistency.
    %\blfootnote{\textbf{2020 Mathematics Subject Classification:} Primary 03E55, 03E25.}
    %\blfootnote{\textbf{Declarations:} This paper forms part of the authors PhD thesis under the supervision of Professor Joan Bagaria at the University of Barcelona.}
    %\blfootnote{\textbf{Acknowledgments:} The author acknowledges partial support from the Generalitat de Catalunya (Catalan Government) under grant 2021 SGR 00348 and from the Spanish Government under grants MTM-PID2020-116773GBI00, PID2023-147428NB-I00, and EUR2022-134032.}
\end{abstract}

%\textbf{Keywords:} Reinhardt Cardinals, Elementary Embeddings, Large Cardinals Beyond Choice, Kunen Inconsistency

%\medskip

%\textbf{Acknowledgments:} The author acknowledges partial support from the Generalitat de Catalunya (Catalan Government) under grant 2021 SGR 00348 and from the Spanish Government under grants MTM-PID2020-116773GBI00, PID2023-147428NB-I00, and EUR2022-134032.

%\medskip

%\textbf{Declarations:} This paper forms part of the authors PhD thesis under the supervision of Professor Joan Bagaria at the University of Barcelona.

%%%%%%%%%%%%%%%%%%%%%%%%%%%%%%  BODY  %%%%%%%%%%%%%%%%%%%%%%%%%%%%%%%%%%%%%%%%%%%%%%%%%%%%%

\section{Introduction}

Let us supplement the usual first-order language of set theory with a functional symbol $j,$ and let $\ZFC(j)$ be the collection of the following axioms:
\begin{enumerate}
    \item The usual axioms of $\ZFC.$
    \item Comprehension and Replacement for formulas in which $j$ appear.
    \item Axioms asserting that $j$ is a nontrivial elementary embedding of the set-theoretic universe $V$ into itself.
\end{enumerate}
The consideration of models of $\ZFC(j)$ was first proposed by William N. Reinhardt in his PhD dissertation in 1967 \cite{reinhardt}.
However, this theory was soon found to be inconsistent by Kenneth Kunen \cite{kunen}. Crucial to his proof was the use of the Axiom of Choice ($\AC$) and it is still unknown if an inconsistency exists without it.

The background theory for this paper will be $\ZFC(j)$ without $\AC,$ which we denote by $\ZF(j).$ We will use $\kappa$ to denote the \emph{Reinhardt cardinal} corresponding to $j,$ that is, the critical point of $j$ (or, in symbols, $\kappa=\crit j).$ In Section~\ref{sec:eventually_dominating_functions}, we prove a result that concerns $j$ and its relation to functions that ``eventually dominate'' $j$ on regular cardinals.

\begin{definition}\label{eventually_dominates}
    Given a limit ordinal $\delta$ and two functions $f,g:\delta\rightarrow\delta,$ we say that $g$ \emph{eventually dominates} $f,$ and write $f\leq^*g,$ iff there exists $\alpha<\delta$ such that $f(\beta)\leq g(\beta)$ for all $\beta\geq\alpha.$
\end{definition}

\newtheorem*{thm:no_dominating_function}{Theorem~\ref{no_dominating_function}}
\begin{thm:no_dominating_function}
    If $\delta>\kappa$ is a regular cardinal such that $j(\delta)=\delta,$ then there is no function $g:\delta\rightarrow\delta$ in the range of $j$ such that $j\vert_{\delta}\leq^* g.$
\end{thm:no_dominating_function}

For any elementary embedding $k:V_\delta\rightarrow V_\delta,$ where $\delta$ is a limit ordinal, we will consider elementary embeddings $l:V_\delta\rightarrow V_\delta$ that are ``roots'' of $k$ (Definition~\ref{def_application_operation}). Concerning this, in Section~\ref{sec:extendibility_behavior}, we prove the following result which shows the existence of an ordinal $\alpha$ that behaves somewhat like an extendible cardinal in $V_\delta.$ Let $\lambda>\kappa$ be the least ordinal such that $j(\lambda)=\lambda.$

\newtheorem*{thm:jumping_roots}{Theorem~\ref{jumping_roots}}
\begin{thm:jumping_roots}
    For all regular cardinals $\delta>\lambda$ such that $j(\delta)=\delta,$ there exists $\alpha<\delta$ such that, for every $\beta\in (\alpha,\delta),$ there is a root $k$ of $j\vert_{V_\delta}$ that satisfies $k(\alpha)>\beta.$
\end{thm:jumping_roots}

In the same section, we introduce the notion of ``$(j,\delta)$-smallness'' (Definition~\ref{def:small_sets}) and establish the following:

\newtheorem*{thm:no_small_sets}{Theorem~\ref{no_small_sets}}
\begin{thm:no_small_sets}
    There is no $(j,\delta)$-small set for any regular cardinal $\delta>\lambda.$
\end{thm:no_small_sets}

Goldberg showed that the existence of a Reinhardt cardinal implies the existence of a proper class of cardinals that are \emph{almost} supercompact \cite{goldberg_measurable_choiceless}. In the same paper, he shows that if $\eta$ is almost supercompact then either $\eta$ or $\eta^+$ is regular.\footnote{Successor cardinals in $\ZF$ need not be regular. See the discussion at the beginning of Section~\ref{sec:regular_cardinals} for more details.} We improve slightly on this result by proving the following in Section~\ref{sec:regular_cardinals}:

\newtheorem*{thm:thm_proper_is_not_regular}{Theorem~\ref{thm_proper_is_not_regular}}
\begin{thm:thm_proper_is_not_regular}
    If $\eta>\lambda$ is an almost supercompact cardinal that is not a limit of almost supercompact cardinals, then it is not regular.
\end{thm:thm_proper_is_not_regular}

Hence, in such cases, it is always $\eta^+$ that is a regular cardinal. Finally, in the last section, we present an alternative proof of the Kunen inconsistency by proving that there are $(j,\delta)$-small sets under $\AC.$

\newtheorem*{thm:small_sets_under_ac}{Theorem~\ref{thm:small_sets_under_ac}}
\begin{thm:small_sets_under_ac}[$\AC$]
    There exists an ordinal $\theta$ such that for every singular almost supercompact $\eta>\theta$ such that $\cof(\eta)=\omega$ and $j(\eta)=\eta,$ there exists a $(j,\eta^+)$-small set.
\end{thm:small_sets_under_ac}

\newtheorem*{cor:kunen_inconsistency}{Corollary~\ref{cor:kunen_inconsistency}}
\begin{cor:kunen_inconsistency}[The Kunen Inconsistency]
    The theory $\ZFC(j)$ is inconsistent.
\end{cor:kunen_inconsistency}

\section{Eventually Dominating Functions}\label{sec:eventually_dominating_functions}

In this short section we just prove Theorem~\ref{no_dominating_function}. The following lemma is easy.

\begin{lemma}\label{lemma_club_not_all_fixed}
    For any cardinal $\delta$ of cofinality strictly greater than $\kappa$ and any club $C\subset \delta$ with increasing enumeration $\langle \alpha_\xi\mid\xi<\cof(\delta)\rangle,$ if $j(\alpha_\xi)=\alpha_\xi$ for all $\xi<\kappa,$ then $j(\alpha_{\kappa})> \alpha_\kappa.$
\end{lemma}

\begin{proof}
    Suppose towards a contradiction that $j(\alpha_\kappa)=\alpha_\kappa.$ Consider $j(\langle \alpha_\xi\mid \xi <\kappa+1\rangle) = \langle \beta_\xi\mid \xi <j(\kappa)+1\rangle.$ By elementarity, $\beta_\xi = j(\alpha_\xi)=\alpha_\xi,$ for all $\xi<\kappa.$ By closure of $j(C),$ we know that $\beta_\kappa=\sup\{\beta_\xi\mid \xi<\kappa\} = \sup\{\alpha_\xi\mid \xi<\kappa\} = \alpha_\kappa.$ So, $j(\alpha_\kappa) = \alpha_\kappa=\beta_\kappa.$ But, $j(\alpha_\kappa) = \beta_{j(\kappa)}>\beta_\kappa.$
\end{proof}

\begin{theorem}\label{no_dominating_function}
    If $\delta>\kappa$ is a regular cardinal such that $j(\delta)=\delta,$ then there is no function $g:\delta\rightarrow\delta$ in the range of $j$ such that $j\vert_{\delta}\leq^* g.$
\end{theorem}

\begin{proof}
Work towards a contradiction and let $j(f)=g$ be such that $j\vert_{\delta}\leq^* g.$ Fix $\alpha<\delta$ such that $j\vert_{\delta}(\beta)\leq g(\beta)$ for all $\beta\geq\alpha.$ Define the sequence $x=\langle x_\xi\mid \xi<\delta \rangle$ from $f$ and $\alpha$ by setting $x_0=\alpha,$ taking limits at limit stages, and at successor stages taking $x_{\xi+1} = f(\beta)$ where $\beta\geq x_{\xi}$ is the least such that $f(\beta)> \beta.$ We need to make sure that this is well defined. The regularity of $\delta$ ensures success at limit stages. For the successor stages we need to check that it is always possible to find a $\beta$ arbitrarily high below $\delta$ such that $f(\beta)>\beta.$ By elementarity of $j,$ we can do this by checking the same for $g,$ and since $j\vert_{\delta}\leq^* g,$ this can be accomplished by making sure that $j\vert_{\delta}$ satisfies that condition. But $j\vert_{\delta}$ clearly satisfies that condition: There are arbitrarily high $\gamma<\delta$ such that $j(\gamma)=\gamma,$ and for any such $\gamma$ we have $j(\gamma+\kappa)=\gamma+j(\kappa)>\gamma+\kappa.$

The sequence $\langle x_\xi\mid \xi<\delta \rangle$ is normal, so it must have unboundedly many fixed points. Let us now consider $x$ and $j(x)=y=\langle y_\xi\mid \xi<\delta\rangle.$ Let $C_x$ and $C_y$ denote the sets of fixed points of $x$ and $y,$ respectively, and let $C_{j\vert_{\delta}}$ denote the set of fixed points of $j\vert_{\delta}.$ As $C_x$ and $C_y$ are clubs and $C_{j\vert_{\delta}}$ is a $<\!\kappa$-club, their intersection $C_x\cap C_y\cap C_{j\vert_{\delta}}$ must also be a $<\!\kappa$-club. The closure of this intersection is a club, which we denote by $C,$ and let $\langle c_\xi\mid \xi<\delta\rangle$ be its increasing enumeration. 

By Lemma~\ref{lemma_club_not_all_fixed}, $j(c_\kappa)>c_\kappa.$ We also have $c_\kappa = x_{c_\kappa} = y_{c_\kappa},$ because $C\subset C_x,C_y.$ Additionally, by applying $j$ to $x_{c_\kappa}=c_\kappa,$ we get $y_{j(c_\kappa)} = j(c_\kappa).$ By definition of $x$ and elementarity of $j,$ $y_{c_\kappa+1} = g(\beta)$ where $\beta\geq y_{c_\kappa}$ is the least such that $g(\beta)>\beta.$ The least such $\beta$ is $y_{c_\kappa},$ because $g(y_{c_\kappa}) \geq j(y_{c_\kappa}) = j(c_\kappa) > c_\kappa = y_{c_\kappa}.$ We now have $y_{c_\kappa+1} = g(y_{c_\kappa}) \geq j(c_\kappa) = j(x_{c_\kappa}) = y_{j(c_\kappa)}.$ But, $j(c_\kappa)>c_\kappa+1$ implies $y_{j(c_\kappa)}>y_{c_\kappa+1},$ a contradiction.
\end{proof}

\section{Extendibility Behavior}\label{sec:extendibility_behavior}

Given a nontrivial elementary embedding $k: V_\delta \rightarrow V_\delta$ (allowing for $\delta=\OR$ here), the \emph{critical sequence} $\langle\kappa_n(k)\mid n\in\omega\rangle$ of $k$ is defined recursively by setting $\kappa_0(k)=\crit k$ and $\kappa_{n+1}(k)= k(\kappa_n(k)).$ The supremum of this sequence will be denoted by $\lambda(k).$ Notice that $\lambda(k)$ is the first fixed point of $k$ above its critical point. We will simplify notation by letting $\kappa_n=\kappa_n(j)$ and $\lambda=\lambda(j).$ 

For every ordinal $\delta,$ let $\mathcal{E}_\delta$ denote the set of all nontrivial elementary embeddings $k:V_\delta\rightarrow V_\delta.$ It is an easy argument to show that $\mathcal{E}_\delta$ is nonempty for all $\delta\geq\lambda:$ If not true, take the least $\delta_0$ counterexample and notice that $j\vert_{V_{\delta_0}}\in \mathcal{E}_{\delta_0}.$

\begin{definition}\label{def_application_operation}
For $\delta$ a limit ordinal and $k,l\in \mathcal{E}_\delta,$ define the operation $k[l],$ the \emph{application of $k$ to $l,$} by setting $k[l] = \bigcup_{\alpha<\delta}k(l\vert_{V_\alpha}).$
\end{definition}

\begin{lemma}\label{application_on_range}
    For $k,l\in \mathcal{E}_\delta$ where $\delta$ is a limit ordinal, $k[l](k(a)) = k(l(a))$ for all $a\in V_\delta.$
\end{lemma}

\begin{proof}
    Fix some $\alpha<\delta$ such that $a\in V_\alpha.$ Then, $k[l](k(a)) = k(l\vert_{V_\alpha})(k(a)) = k(l\vert_{V_\alpha}(a)) = k(l(a)).$
\end{proof}

The following lemma is similar to \cite[Lemma~1.6]{dehornoy}.

\begin{lemma}\label{prop_application_op_and_critical_point}
If $k,l\in \mathcal{E}_\delta$ where $\delta$ is a limit ordinal, then $k[l]$ is also in $\mathcal{E}_\delta.$ Moreover, $\crit k[l] = k(\crit{l}),$ and if $\langle\gamma_n\mid n\in \omega\rangle$ is the critical sequence of $l,$ then $\langle k(\gamma_n)\mid n\in\omega\rangle$ is the critical sequence of $k[l].$
\end{lemma}

\begin{proof}
First note that, for any $\alpha_1,\alpha_2<\delta,$ the fact that the two functions $l\vert_{V_{\alpha_1}}$ and $l\vert_{V_{\alpha_2}}$ are compatible implies that $k(l\vert_{V_{\alpha_1}})$ and $k(l\vert_{V_{\alpha_2}})$ are compatible. Therefore, $k[l]$ is a function with domain and codomain $V_\delta.$ Also, it is injective since it is the union of a $\subset$-chain of injections.

To see that it is elementary, fix a formula $\phi(x)$ and an ordinal $\alpha<\delta.$ By elementarity of $l,$ we have
$$\forall x\in V_\alpha (V_\delta\satisfies \phi(x) \iff V_\delta\satisfies \phi(l\vert_{V_\alpha}(x))).$$
Applying $k$ to the above formula gives
$$\forall x\in V_{k(\alpha)} (V_\delta\satisfies \phi(x) \iff V_\delta \satisfies \phi(k(l\vert_{V_\alpha})(x))).$$ Since $\alpha$ was arbitrary, the above must be correct for all $x$ in $V_\delta.$

The fact that $\crit{k[l]} = k(\crit{l})$ follows from two facts: $k[l](k(\crit{l}))>k(\crit{l}),$ which follows from $l(\crit{l})>\crit{l},$ and $\forall \alpha < k(\crit{l}) (k[l](\alpha) = \alpha),$ which follows from $\forall \alpha < \crit{l}(l(\alpha)= \alpha).$
For the final claim of the lemma:
\begin{multline*}
    k[l]^{n}(\crit k[l]) = k[l]^{n}(k(\crit{l})) = k[l]^{n-1}(k[l](k(\crit{l})))
    \\ = k[l]^{n-1}(k(l (\crit l))) = k(l^n(\crit l)),
\end{multline*}
by $n$ applications of Lemma~\ref{application_on_range}.
\end{proof}

Henceforward, we will always assume that $\delta$ is a limit ordinal.

\begin{definition}\label{def_iterations_and_roots}
    For any $k\in \mathcal{E}_\delta,$ we can define the two sets $I(k)=\{k_n\mid n\geq1\},$ where $k_1=k$ and $k_{n+1}=k_n[k_n],$ and $R(k) = \{l\in \mathcal{E}_\delta\mid l[l]=k\}.$ Whenever $l[l]=k,$ we will call $l$ a \emph{root} of $k$ and $k$ \emph{the square} of $l.$
\end{definition}

\begin{definition}
    For $k\in \mathcal{E}_\delta$ define the set $A(k)$ by the following recursion: Set $A_0=I(k)$ and $A_{n+1}= A_n\cup \bigcup_{l\in A_n} R(l),$ and let $A(k)=\bigcup_n A_n.$
\end{definition}

Notice that the set $A(k)$ is the smallest set containing $k$ and closed under taking squares and roots.

\begin{lemma}\label{prop_fixing_A}
    If $\delta>\kappa$ is a limit ordinal such that $j(\delta)=\delta,$ then $j(A(j\vert_{V_\delta}))= A(j\vert_{V_\delta}).$
\end{lemma}

\begin{proof}
    Denote $j\vert_{V_\delta}$ by $j'$ for simplicity.
    First, by elementarity of $j,$ we have $j(A(j'))=A(j(j')).$ Then, noticing that $j'= \bigcup_{\alpha<\delta}j'\vert_{V_\alpha},$ we get
    $$
    j(j') = j\Big(\bigcup_{\alpha<\delta}j'\vert_{V_\alpha}\Big) = \bigcup_{\alpha<\delta}j(j'\vert_{V_\alpha}) = \bigcup_{\alpha<\delta}j'(j'\vert_{V_\alpha}) = j'[j'],
    $$
    hence $A(j(j'))=A(j'[j']).$ Finally, we establish $A(j'[j'])=A(j'):$ $I(j'[j'])\subset I(j')$ implies $A(j'[j'])\subset A(j')$ by definition. For the reverse inclusion, notice that $j'\in R(j'[j'])\subset A(j'[j']),$ and since $A(j')$ is the smallest set containing $j'$ and closed under taking squares and roots, it must be that $A(j')\subset A(j'[j']).$ Putting everything together, we have
    \[
    j(A(j'))= A(j(j'))=A(j'[j'])=A(j').\qedhere
    \]
\end{proof}

\begin{theorem}\label{jumping_embeddings}
    For all regular cardinals $\delta>\lambda$ such that $j(\delta)=\delta,$ there exists $\alpha<\delta$ such that, for every $\beta\in (\alpha,\delta),$ there is a $k\in A(j\vert_{V_\delta})$ satisfying $k(\alpha)>\beta.$
\end{theorem}

\begin{proof}
    Let $A$ denote $A(j\vert_{V_\delta})$ for simplicity. Working towards a contradiction, fix any such $\delta$ and suppose there is no such $\alpha<\delta.$ Define $f:\delta\rightarrow\OR$ by setting $f(\xi) = \sup\{k(\xi)\mid k\in A\}.$ By assumption, $f(\xi)<\delta,$ for all $\xi<\delta.$ Since by Lemma~\ref{prop_fixing_A} $j(A)=A,$ we must have $j(f)=f.$ But clearly, $f\geq^* k\vert_{\delta}$ for all $k\in A,$ and in particular, $f\geq^* j\vert_{\delta},$ contradicting Theorem~\ref{no_dominating_function}.
\end{proof}

Thus $\alpha$ behaves somewhat similar to extendible cardinals inside $V_\delta.$ Such behaviour in a more global form under $\ZF$ alone is already considered in Goldberg \cite{goldberg_measurable_choiceless}, Asper\'o \cite{aspero}, and Mohammd \cite{marwan}

We can impose even further restrictions on the elementary embeddings $k$ above while still getting the same result:

\begin{theorem}\label{jumping_roots}
    For all regular cardinals $\delta>\lambda$ such that $j(\delta)=\delta,$ there exists $\alpha<\delta$ such that, for every $\beta\in (\alpha,\delta),$ there is a $k\in R(j\vert_{V_\delta})$ satisfying $k(\alpha)>\beta.$
\end{theorem}

\begin{proof}
    First notice that $R(j\vert_{V_\delta})$ is not empty since $j(R(j\vert_{V_\delta})) = R(j(j\vert_{V_\delta})) = R(j\vert_{V_\delta}[j\vert_{V_\delta}])$ is not, as witnessed by $j\vert_{V_\delta}.$ This time define $f:\delta\rightarrow\OR$ by setting $f(\xi) = \sup\{k(\xi)\mid k\in R(j\vert_{V_\delta})\}.$ Again, if the theorem fails for $\delta,$ then $f(\xi)<\delta$ for all $\xi<\delta.$ Clearly $f\geq^* k\vert_{\delta}$ for all $k\in R(j\vert_{V_\delta}).$ By elementarity, $j(f)\geq^* k\vert_{\delta}$ for all $k\in j(R(j\vert_{V_\delta})) = R(j(j\vert_{V_\delta})) = R(j\vert_{V_\delta}[j\vert_{V_\delta}]).$ In particular, $j(f)\geq^* j\vert_{\delta},$ contradicting Theorem~\ref{no_dominating_function}.
\end{proof}

The above theorem can be proven for any set $X\subset \mathcal{E}_\delta,$ satisfying $j\vert_{V_\delta}\in j(X),$ in place of $R(j\vert_{V_\delta}).$

\begin{definition}\label{def:small_sets}
    Given a regular cardinal $\delta>\lambda$ such that $j(\delta)=\delta$ and a set $X\subset \mathcal{E}_\delta,$ we say that $X$ is \emph{$(j,\delta)$-small} iff $j\vert_{V_\delta}\in j(X)$ and $\sup\{k(\xi)\mid k\in X\}<\delta,$ for all $\xi<\delta.$ 
\end{definition}

Thus, we have the following:

\begin{theorem}\label{no_small_sets}
    There is no $(j,\delta)$-small set for any regular cardinal $\delta>\lambda.$\qed
\end{theorem}

We will show in the last section that $\AC$ implies the existence of $(j,\delta)$-small sets for unboundedly many $\delta,$ which will give us the Kunen inconsistency.

\section{Regular Cardinals}\label{sec:regular_cardinals}

In the context of Choice, it is a basic set theoretic fact that every successor cardinal is regular. In the absence of $\AC,$ there is no guarantee that successor cardinals are regular. In fact, Moti Gitik has showed that it is consistent with $\ZF$ that there are no regular uncountable cardinals \cite{gitik}.

In $\ZF(j),$ we already know that the $\kappa_n$ are regular for all $n\in \omega.$ David Asper\'o asked whether there are regular cardinals above $\lambda,$ and Goldberg answered this question positively in \cite{goldberg_measurable_choiceless}. We will need a more detailed account of Goldberg's result, so let us start by recalling what is necessary from his paper.

\begin{definition}[\cite{goldberg_measurable_choiceless}]
    A cardinal $\eta$ is said to be \emph{$(\gamma,\nu, x)$-almost supercompact} for $\gamma<\eta<\nu$ and $x\in V_\nu$ iff there exists $\bar \nu<\eta$ and $\bar x \in V_{\bar \nu}$ for which there is an elementary embedding $k:V_{\bar\nu}\rightarrow V_\nu$ such that $k(\gamma)=\gamma$ and $k(\bar x)=x.$
    We say that $\eta$ is \emph{$<\!\mu$-almost supercompact} iff $\eta$ is $(\gamma,\nu, x)$-almost supercompact for all $\gamma<\eta<\nu<\mu$ and all $x\in V_\nu,$ and we simply say that $\eta$ is almost supercompact iff it is $<\!\mu$-almost supercompact for all $\mu>\eta.$
\end{definition}

\begin{figure}[htbp]
    \centering
    \begin{minipage}{0.49\textwidth}
        \centering
        \includegraphics{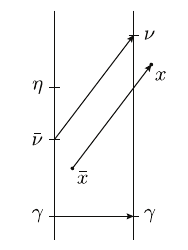}
        \captionof{figure}{Almost supercompactness}
        \label{fig_almost_supercompactness}
    \end{minipage}
    \begin{minipage}{0.49\textwidth}
        \centering
        \includegraphics{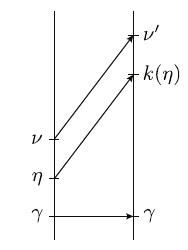}
        \captionof{figure}{Almost extendibility}
        \label{fig_almost_extendibility}
    \end{minipage}
\end{figure}

\begin{definition}[\cite{goldberg_measurable_choiceless}]
    A cardinal $\eta$ is said to be \emph{$(\gamma,\nu)$-almost extendible} for $\gamma<\eta<\nu$ iff there is an elementary embedding $k:V_\nu \rightarrow V_{\nu'}$ such that $k(\gamma)=\gamma$ and $k(\eta)>\nu.$
    We say that $\eta$ is \emph{$<\!\mu$-almost extendible} iff $\eta$ is $(\gamma,\nu)$-almost extendible for all $\gamma<\eta<\nu<\mu,$ and we simply say that $\eta$ is \emph{almost extendible} iff it is $<\!\mu$-almost extendible for all $\mu>\eta.$
\end{definition}

Notice that a cardinal can be almost supercompact simply by the virtue of being a limit of almost supercompact cardinals, and the same is true for almost extendible cardinals. Thus the two classes of almost supercompact cardinals and almost extendible cardinals are closed.

\begin{proposition}[\cite{goldberg_measurable_choiceless}]\label{cor_extendible_to_supercompact}
    If a cardinal $\eta$ is $<\!\mu$-almost extendible, where $\mu$ is a limit ordinal, then it is also $<\!\mu$-almost supercompact.\qed
\end{proposition}

\begin{proposition}[\cite{goldberg_measurable_choiceless}]\label{prop_reinhardt_to_club_extendibles}
    If there is a Reinhardt cardinal, then there is a club proper class of almost extendible cardinals.\qed
\end{proposition}

Goldberg proves that if $\eta$ is almost supercompact, then every successor cardinal greater than $\eta$ has cofinality at least $\eta$ \cite[Corollary~2.18]{goldberg_measurable_choiceless}. This, in turn, implies the following proposition, which along with Propositions \ref{cor_extendible_to_supercompact} and \ref{prop_reinhardt_to_club_extendibles} give a proper class of regular cardinals in $\ZF(j).$

\begin{proposition}[\cite{goldberg_measurable_choiceless}]\label{cor_goldberg_regular}
    If $\eta$ is almost supercompact, then either $\eta$ or $\eta^+$ is a regular cardinal.\qed
\end{proposition}

Let us call an almost supercompact cardinal that is not a limit of almost supercompact cardinals a \emph{successor almost supercompact}. We will prove that if $\eta$ is a successor almost supercompact cardinal above $\lambda,$ then it cannot be regular. Thus, by Goldberg's result, we must have $\eta^+$ regular for all $\eta>\lambda$ that is a successor almost supercompact cardinal. We will need a few intermediary results first. The following lemma is easy.

\begin{lemma}\label{gluing_almost_supercompacts}
    If $\eta_0$ is $<\!\eta_1$-almost supercompact and $\eta_1$ is almost supercompact, then $\eta_0$ is also almost supercompact.
\end{lemma}

\begin{proof}
    Fix $\gamma<\eta_0,$ some $\nu\geq\eta_1,$ and $x\in V_{\nu}.$ Let $\gamma'$ code the pair $\langle\gamma,\eta_0\rangle$ in the canonical well-ordering of $\OR\times \OR.$ Notice that $\gamma'<\eta_0<\eta_1,$ so by almost supercompactness of $\eta_1,$ we can find an elementary embedding $k:V_{\bar\nu+\omega}\rightarrow V_{\nu+\omega}$ such that $\bar\nu+\omega<\eta_1,\ k(\gamma')=\gamma',\ k(\langle\bar\gamma,\bar\eta_0\rangle)=\langle\gamma,\eta_0\rangle,$ and $k(\bar x)=x,$ for some $\langle\bar\gamma,\bar\eta_0\rangle,\bar x\in V_{\bar\nu}.$ Since the canonical well-ordering of $\OR\times\OR$ is $\Delta_0,$ elementarity of $k$ and the fact that $k(\gamma')=\gamma'$ imply that $\langle\bar\gamma,\bar\eta_0\rangle = \langle\gamma,\eta_0\rangle.$ Hence, both $\gamma$ and $\eta_0$ are fixed by $k.$ Since $\eta_0<\eta_1\leq\nu,$ we must also have $\eta_0<\bar\nu$ by elementarity.

    We already know that $\bar\nu<\eta_1,$ so we can now use $<\!\eta_1$-almost supercompactness of $\eta_0$ to get another elementary embedding $l:V_{\dbar{\nu}}\rightarrow V_{\bar \nu}$ such that $\dbar{\nu}<\eta_0, \ l(\gamma)=\gamma,$ and $l(\dbar{x})=\bar x.$ Finally, the composite elementary embedding $k\circ l: V_{\dbar{\nu}}\rightarrow V_\nu$ fixes $\gamma$ and has $x$ in its range, and is therefore our witness.
\end{proof}

Definition~\ref{def_application_operation} and Lemma~\ref{prop_application_op_and_critical_point} also work for $\delta=\OR,$ i.e., for elementary embeddings $j:V\rightarrow V.$ Thus, let $j_n, n<\omega,$ be defined from $j$ as in Definition~\ref{def_iterations_and_roots}.

\begin{lemma}[{\cite[Lemma~1.3]{schlutzenberg}}]\label{iterate_to_fix}
    For all $\alpha,$ there exists $n$ such that $j_n(\alpha)=\alpha.$
\end{lemma}

\begin{proof}
Suppose this is not the case and fix $\alpha$ as the least counterexample. Let $\gamma>\alpha$ be a limit ordinal fixed by $j,$ and denote $j\vert_{V_\gamma}$ by $j'.$
First, consider the sequence $j(\langle j'_1,j'_2,j'_3\cdots\rangle)=\langle j(j'_1),j(j'_2),j(j'_3)\cdots\rangle.$ We know that $j(j'_1)=j'_2$ by definition, hence, by elementarity of $j,$ we must have $j(j'_2)=j'_2[j'_2]=j'_3, j(j'_3)=j'_3[j'_3]=j'_4,\ldots$ and so on. Thus, $j(j'_n)=j'_{n+1}$ for all $n\geq 1.$

Now, for each $n\geq 1,$ let $A_n\subset \alpha$ be the set of ordinals in $\alpha$ that are fixed by $j'_n.$ By minimality of $\alpha,$ it must be that $\alpha=\bigcup_{n}A_n.$ Now, $j(A_n)$ is the set of ordinals in $j(\alpha)$ that are fixed by $j(j'_n)=j'_{n+1},$ and $j(\alpha) = j(\bigcup_{n}A_n) = \bigcup_n j(A_n).$ Since, by our assumption, $\alpha\in j(\alpha),$ there must be some $m$ such that $\alpha\in j(A_m).$ But, this means that $\alpha$ is a fixed point of $j'_{m+1},$ contradicting the choice of $\alpha.$
\end{proof}

Using the lemma above and coding finite sequences of ordinals as single ordinals, we can always find some $n$ such that $j_n$ fixes any desired finite set of ordinals. Given an elementary embedding $k:V_\delta\rightarrow V_\delta$ and $\gamma<\delta,$ let $R^\gamma(k)=\{l\in R(k)\mid l(\gamma)=\gamma\}.$

\begin{theorem}\label{thm_roots_fixing_gamma}
    Let $\delta>\lambda$ be a regular cardinal, let $\gamma<\delta,$ and let $n$ be such that $j_n$ fixes both $\gamma$ and $\delta.$ Then there exists $\alpha<\delta$ such that for all $\beta>\alpha$ there is $k\in R^\gamma(j_n\vert_{V_\delta})$ such that $k(\alpha)>\beta.$
\end{theorem}

\begin{proof}
    Similar to the proof of Theorem~\ref{jumping_roots} using $j_n$ in place of $j.$
\end{proof}

\begin{corollary}\label{cor_local_almost_extendibles}
    For any regular limit cardinal $\delta>\lambda,$ $V_\delta$ satisfies that there exists a club class of almost extendible cardinals.
\end{corollary}

\begin{proof}
    For each $\gamma<\delta,$ let $\alpha_\gamma<\delta$ be the least ordinal $\alpha$ given by Theorem~\ref{thm_roots_fixing_gamma}. Define $F:\delta\rightarrow\delta$ by setting $F(\xi)=\sup\{\alpha_\gamma\mid\gamma<\xi\}.$ This is welldefined by the regularity of $\delta.$ Notice that $F$ is an increasing and continuous function, so the set of fixed points of $F$ form a club $C\subset \delta.$ If $D\subset\delta$ is the club of all cardinals below $\delta,$ then clearly any member of $C\cap D$ is an almost extendible cardinal in $V_\delta.$
\end{proof}

We are now ready to prove the main theorem of this section.

\begin{theorem}\label{thm_proper_is_not_regular}
    If $\eta>\lambda$ is a successor almost supercompact cardinal, then it is not regular.
\end{theorem}

\begin{proof}
    Let $\alpha<\eta$ be such that there is no almost supercompact cardinal in the open interval $(\alpha,\eta)$ and assume towards a contradiction that $\eta$ is regular. It is easy to see that any almost supercompact cardinal is a limit cardinal, so Corollary~\ref{cor_local_almost_extendibles} applies to $\eta,$ and we can fix a cardinal $\beta\in (\alpha,\eta)$ that is almost extendible in $V_\eta.$ This means that $\beta$ is $<\! \eta$-almost extendible, and hence $<\!\eta$-almost supercompact by Proposition~\ref{cor_extendible_to_supercompact}. Now, $\beta$ must be almost supercompact by Proposition~\ref{gluing_almost_supercompacts}, a contradiction.
\end{proof}

By Goldberg's result, Proposition~\ref{cor_goldberg_regular}, we obtain the following:

\begin{corollary}\label{suc_of_proper_almost_sup_is_regular}
    If $\eta>\lambda$ is a successor almost supercompact cardinal, then $\eta^+$ is regular.
\end{corollary}

\section{Small Sets Under $\AC$}\label{sec:small_sets_under_ac}

In this section we present an alternative proof of the Kunen inconsistency. In particular, we will prove that $\AC$ implies the existence of small sets, which contradicts Theorem~\ref{no_small_sets}. The proof is based on the proof of Solovay's result, given in \cite[Theorem~20.8]{jech}, showing that the singular cardinal hypothesis holds above a strongly compact cardinal.

\begin{lemma}[$\AC$]
    There exists an ordinal $\theta$ such that for every singular almost supercompact $\eta>\theta,$ there exists an almost supercompact $\eta'<\eta$ and a collection $\{M_\alpha \subset \eta^+\mid \alpha< \eta^+\}$ such that $|M_\alpha|<\eta'$ and 
    $$
    [\eta^+]^\omega=\bigcup_{\alpha<\eta^+} [M_\alpha]^\omega.
    $$
\end{lemma}

\begin{proof}
    Suppose otherwise and let $\langle \eta_\xi\mid \xi\in\OR\rangle$ be an increasing enumeration of singular almost supercompact cardinals above $\lambda$ for which this fails. By Proposition~\ref{cor_goldberg_regular}, $\eta_\xi^+$ is regular for every $\xi.$ Notice that
    $$
    \eta_\kappa< \eta_{\kappa+1} <\eta_{\kappa+1}^+<j(\eta_\kappa)=\eta_{j(\kappa)}<\sup j"\eta_{\kappa+1}^+<j(\eta_{\kappa+1}^+).
    $$
    Let $\eta'=\eta_\kappa, \eta=\eta_{\kappa+1},$ and $\sigma=\sup j"\eta_{\kappa+1}^+.$ Define the $\kappa$-complete ultrafilter $D$ on $\eta^+$ by setting $X\in D \iff \sigma\in j(X),$ for all $X\subset \eta^+.$ Since $\cof(\sigma)=\eta^+<j(\eta'),$ the set $E=\{\alpha<\eta^+\mid \cof(\alpha)<\eta'\}$ belongs to $D.$ 

    Using $\AC,$ for each $\alpha\in E,$ fix $A_\alpha\subset \alpha$ cofinal with $|A_\alpha|<\eta'.$ If $\alpha<\eta^+$ is not in $E,$ then set $A_\alpha=\emptyset.$ Let $\langle B_\alpha\mid\alpha< j(\eta^+)\rangle=j(\langle A_\alpha\mid \alpha<\eta^+\rangle).$ Since  $B_\sigma$ is cofinal in $\sigma=\sup j" \eta^+,$ for every $\mu<\eta^+,$ there is $\mu'\in(\mu,\eta^+)$ such that $[j(\mu),j(\mu'))\cap B_\sigma\neq \emptyset.$ Define the sequence $\langle \mu_\zeta\mid \zeta<\eta^+\rangle$ in $\eta^+$ recursively by setting $\mu_0=0,$ taking limits at limit stages, and at successor stages taking $\mu_{\zeta+1}<\eta^+$ to be such that $[j(\mu_\zeta),j(\mu_{\zeta+1}))\cap B_\sigma\neq\emptyset.$ For $\zeta<\eta^+,$ set $I_\zeta=[\mu_\zeta,\mu_{\zeta+1}).$ For each $\alpha<\eta^+,$ define
    $$
    M_\alpha=\{\zeta<\eta^+\mid I_\zeta\cap A_\alpha\neq \emptyset\}.
    $$
    Now, fix any $\zeta<\eta^+.$ By construction $j(I_\zeta)=[j(\mu_\zeta),j(\mu_{\zeta+1}))$ intersects $B_\sigma.$ Thus, the set of all $\alpha<\eta^+$ such that $\zeta\in M_\alpha$ belongs to $D.$

    We will show that $\{M_\alpha\mid \alpha<\eta^+\}$ and $\eta'$ witness the conclusion of the lemma for $\eta,$ thereby arriving at a contradiction. First of all, for each $\alpha,$ $|M_\alpha|\leq |A_\alpha|<\delta$ since the $I_\zeta$ are mutually disjoint.
    Next, fix $x\in[\eta^+]^\omega.$ For each $\zeta\in x,$ the set of $\alpha$ such that $\zeta\in M_\alpha$ is in $D.$ Therefore, by $\kappa$-completeness of $D,$ $x\subset M_\alpha$ for some $\alpha.$ Thus, $x\in [M_\alpha]^\omega,$ and we are done.
\end{proof}

\begin{lemma}[$\AC$]
    There exists an ordinal $\theta$ such that for every singular almost supercompact $\eta>\theta$ with countable cofinality we have $|V_{\eta+1}|=\eta^+.$
\end{lemma}

\begin{proof}
    Fix $\theta$ as in the previous lemma. Fix $\eta>\theta,$ and let $\eta'<\eta$ and $\{M_\alpha\mid \alpha< \eta^+\}$ be again as in the previous lemma. Notice that $|V_\eta|=\eta.$ We now have
    \begin{multline*}
        |V_{\eta+1}|= 2^{|V_\eta|}=2^\eta=\eta^\omega\leq (\eta^+)^\omega =|[\eta^+]^\omega| = |\bigcup_{\alpha<\eta^+}[M_\alpha]^\omega|\\
        \leq \sum_{\alpha<\eta^+} |[M_\alpha]^\omega|=\eta^+\cdot \sup_{\alpha<\eta^+} |[M_\alpha]^\omega| \leq \eta^+\cdot \eta=\eta^+,
    \end{multline*}
    where the last inequality follows from the facts that $\eta'$ is a strong limit and that $|M_\alpha|<\eta'$ for all $\alpha<\eta^+.$
\end{proof}

\begin{theorem}[$\AC$]\label{thm:small_sets_under_ac}
    There exists an ordinal $\theta$ such that for every singular almost supercompact $\eta>\theta$ such that $\cof(\eta)=\omega$ and $j(\eta)=\eta,$ there exists a $(j,\eta^+)$-small set.
\end{theorem}

\begin{proof}
    Fix $\theta$ as in the previous lemma and let $\eta>\theta$ be any singular almost supercompact such that $\cof(\eta)=\omega$ and $j(\eta)=\eta.$ By the previous lemma, we can fix a surjection $b:\eta^+\rightarrow \mathcal{E}_{\eta}\subset V_{\eta+1}.$ Let $\beta<\eta^+$ be such that $j\vert_{V_{\eta}}\in \range j(b\vert_{\beta}).$ Let $X\subset \mathcal{E}_{\eta^+}$ consist of all those $k$ such that $k\vert_{V_\eta}\in \range b\vert_{\beta}.$ We will show that $X$ is $(j,\eta^+)$-small.
    
    First, since $j\vert_{V_{\eta}}\in \range j(b\vert_{\beta}),$ we have $j\vert_{V_{\eta^+}}\in j(X).$ Next, we need to show that $\sup\{k(\xi)\mid k\in X\}<\eta^+,$ for all $\xi<\eta^+.$ Notice that, for $k,l\in\mathcal{E}_{\eta^+},$ $k\vert_{V_{\eta}}= l\vert_{V_{\eta}}$ implies $k\vert_{V_{\eta+1}}=l\vert_{V_{\eta+1}}.$ This is because $k(A)=\bigcup_{\alpha<\eta} k(A\cap V_\alpha),$ for all $A\subset V_\eta$ and all $k\in \mathcal{E}_{\eta^+}.$ Also, $k\vert_{V_{\eta+1}}=l\vert_{V_{\eta+1}}$ implies $k\vert_{\eta^+}=l\vert_{\eta^+},$ since each $\alpha\in(\eta,\eta^+)$ corresponds to some wellordering of $\eta.$ Therefore, for each $\xi<\eta^+,$
    \[
    |\{k(\xi)\mid k\in X\}| \leq |\{k\vert_{\eta^+}\mid k\in X\}| \leq |\{k\vert_{V_\eta}\mid k\in X\}| \leq |(b\vert_{\beta})|<\eta^+.
    \qedhere
    \]
\end{proof}

The Kunen inconsistency now follows as a corollary of the theorem above and Theorem~\ref{no_small_sets}:

\begin{corollary}[The Kunen inconsistency]\label{cor:kunen_inconsistency}
    The theory $\ZFC(j)$ is inconsistent.
\end{corollary}

%%%%%%%%%%%%%%%%%%%%%%%%%%% BIB %%%%%%%%%%%%%%%%%%%%%

\printbibliography

\end{document}